\documentclass[11pt]{article}

\usepackage{todonotes}

\usepackage[top=3.5cm,bottom=3.5cm,left=3.5cm,right=3cm]{geometry}
\pdfpagewidth\paperwidth
\pdfpageheight\paperheight
\usepackage[english,italian]{babel}
\usepackage[T1]{fontenc}
\usepackage[latin1]{inputenc} 
\usepackage{lmodern}
\usepackage{amsfonts}
\usepackage{cancel}
\usepackage{amsmath}
\usepackage{amsthm}
\usepackage{amssymb}
\usepackage{graphicx}
\usepackage{sidecap}
\usepackage{caption}
\usepackage{subfig}
\usepackage{wrapfig}
\usepackage{psfrag}
\usepackage{mathrsfs}
\usepackage{tikz}
\usepackage{multicol}
\usepackage{pgfplots}
\usepackage{hyperref,enumitem}

\newtheorem{teo}{Theorem}
\newtheorem{lem}[teo]{Lemma}

\theoremstyle{definition}

\newtheorem{ex}[]{Example}
\newtheorem*{exSIM}{Example~\ref{ex:SIM} {\normalfont (continued)}}

\theoremstyle{remark}
\newtheorem{oss}[teo]{Remark}
\newtheorem*{oss*}{Remark}

\newcommand{\N}{\mathbb N}
\newcommand{\R}{\mathbb R}
\newcommand{\C}{\mathbb C}

\providecommand{\scal}[2]{\langle#1,#2\rangle}

\newcommand{\cA}{{\mathcal A}}

\newcommand{\cD}{{\mathcal D}}

\newcommand{\cF}{{\mathcal F}}
\newcommand{\cG}{{\mathcal G}}
\newcommand{\cH}{{\mathcal H}}
\newcommand{\cI}{{\mathcal I}}
\newcommand{\cJ}{{\mathcal J}}

\newcommand{\cQ}{{\mathcal Q}}
\newcommand{\cR}{{\mathcal R}}
\newcommand{\cS}{{\mathcal S}}

\newcommand{\cW}{{\mathcal W}}

\newcommand{\email}[1]{\protect\href{mailto:#1}{#1}}

\newcommand{\D}{{\rm d}}
\newcommand{\E}{{\rm e}}

 \newcommand{\gray}[1]{{\color{gray}#1}}

 \newcommand{\norm}[1]{{\left\Vert {#1}\right\Vert}}
 \newcommand{\nt}{w}
 \newcommand{\ct}{c}

\begin{document}
\selectlanguage{english}
\title{ Unitarization and Inversion Formulae for\\ the Radon Transform between Dual Pairs}

\author{Giovanni S.\ Alberti\thanks{Department of Mathematics, University of Genoa, Via Dodecaneso 35, 16146 Genova, Italy (\email{alberti@dima.unige.it}, \email{bartolucci@dima.unige.it}, \email{demari@dima.unige.it}, \email{devito@dima.unige.it}).}
\and  Francesca Bartolucci\footnotemark[1]
\and Filippo De Mari\footnotemark[1]
\and Ernesto De Vito\footnotemark[1]}

\maketitle

\abstract{We consider the Radon transform associated to dual
    pairs $(X,\Xi)$ in the sense of Helgason, with $X=G/K$ and
    $\Xi=G/H$, where  $G=\R^d\rtimes K$,  $K$ is a  closed
      subgroup of ${\rm GL}(d,\R)$ and $H$ is a closed subgroup of
      $G$. Under some technical assumptions, we   prove that if the quasi regular
    representations  of $G$ acting on $L^2(X)$ and $L^2(\Xi)$ are irreducible,
    then the Radon transform admits a unitarization  intertwining the two 
representations.  If, in addition, the representations are square integrable, we provide
an inversion formula for the Radon transform based on the voice transform associated
to these representations.}

\vspace{2mm}

\section{Introduction}

In a remarkable series of papers (see, e.g., \cite{helgason65,helgason65_2}),
for the most part subsumed in the monographs \cite{helgason84,helgason94,helgason99,helgason11},
 Helgason has developed a broad theory of Radon transforms in a differential
geometric setup. In this paper we show that the above framework is particularly appropriate in order to
treat in a unified way some results concerning unitarizability
features and inversion formulae of various types of Radon transforms \cite{bardemadeviodo,
  holschneider91} and permits to handle a significant number
of other  interesting examples.

One of the central notions in Helgason's theory is that of dual
  pair $(G/K,G/H)$ of homogeneous spaces of the same Lie group $G$, 
where $K$ and $H$ are closed subgroups of $G$. The transitive $G$-space $X=G/K$ is
meant to describe the ambient in which the functions to be analyzed
live, prototypically a space of constant curvature like the Euclidean
plane, or the sphere $S^2$ or the hyperbolic plane ${\mathbb H}^2$. A large and
important part of Helgason's work is devoted to the case when $X$ is
actually a symmetric space, whence the notation $G/K$ that we
retain. The second transitive $G$-space  $\Xi=G/H$ is meant to parametrize the set of
submanifolds of $X$ over which one wants to integrate functions, for
instance hyperplanes in Euclidean space, great circles in $S^2$,
geodesics or horocycles in ${\mathbb H}^2$.  With this basic
understanding in mind, the notion
of incidence between $x\in X$ and $\xi\in\Xi$ translates the
intuition that $x=g_1K$ is a point of $\xi=g_2H$ and amounts to the
fact that $g_1K\cap g_2H\neq\emptyset$. In this way any element
$\xi\in\Xi$ may be realized as a submanifold $\hat\xi\subset X$ simply
by taking all the points $x\in X$ that are incident to $\xi$;
conversely, one builds the ``sheaf of manifolds'' $\check x$ through
the point $x\in X$ by taking all the points $\xi\in\Xi$ that are
incident to $x$. If the maps $\xi\mapsto \hat\xi$ and
$x\mapsto\check x$ are injective, then $(X,\Xi)$ is a dual pair. Under
this assumption, the Radon transform $\mathcal{R}$ takes functions on
$X$ into functions on $\Xi$ and is abstractly defined by
\[ 
\mathcal{R}f(\xi)=\int_{\hat{\xi}}\,f(x) {\rm d}m_\xi(x),
\]
provided that, for all $\xi\in \Xi$, $m_\xi$ is a suitable measure on
the manifold 
$\hat{\xi}$ and the right hand side is meaningful, possibly in some weak
sense. The first requirement
is achieved by observing that, denoted by $\xi_0=eH$ the
origin of $\Xi$, it is easy to check that $\widehat{\xi_0}\subset X$ is actually a
transitive $H$-space, hence $\widehat{\xi_0}$ carries a measure $m_0$ which is quasi-invariant
with respect to the $H$-action \cite{varadarajan85}. Via the
$G$-action, for all $\xi\in\Xi$ the measures  $\{m_\xi\}$ on
$\{\hat\xi\}$
are defined as the push-forward of $m_0$.  

As for the right space of functions  $f\colon X\to\C$ for which 
the Radon transform makes sense, a natural choice is the  $L^2$
setting. Indeed,  both $X$ and $\Xi$ are transitive spaces,  so that
there exist quasi-invariant measures ${\rm d}x$ and ${\rm d}\xi$.
In this context, a central issue is to prove that the Radon transform, up
to a composition with a suitable pseudo-differential operator, can be
extended to a unitary map $\cQ$ from $L^2(X,{\rm d}x)$ to
$L^2(\Xi,{\rm d}\xi)$ intertwining the quasi-regular representations $\pi$ and
$\hat{\pi}$ of $G$, which naturally act on $L^2(X,{\rm d}x)$ and
$L^2(\Xi,{\rm d}\xi)$, respectively. We stress that in Helgason's approach, it is assumed that
$m_0$, ${\rm d}x$ and ${\rm d}\xi$ are all invariant measures.  The
reader is again referred to Helgason's books for a thorough treatment, as
well as for the broad problem of the operator properties of
$\mathcal{R}$.

In this paper we address the special case in which the group $G$ is a
semidirect product of the form $\R^d\rtimes K$, 
where $K$ is a closed subgroup of ${\rm GL}(d,\R)$, and the
representations $\pi$ and $\hat{\pi}$ are both irreducible. Under some
technical assumptions 
that we describe below, we prove a unitarization result,
see~Theorem~\ref{ernestoduflomoore}. The proof is based on the 
generalization of Shur's lemma provided  by Duflo and Moore \cite{dumo76}. One of the
novelties of our treatment consists in making weaker assumptions on
$m_0$, ${\rm d}x$ and ${\rm d}\xi$, namely their relative
invariance instead of invariance.  This allows for considering a wider variety of cases, such
as wavelets and shearlets.
A well-known predecessor of Theorem~\ref{ernestoduflomoore} is
Theorem~4.1 in \cite{helgason99}, an alternative proof of which,
tailored to our particular viewpoint, is to be found
in~\cite{bardemadeviodo}.

If, in addition, we require that $\pi$ is square integrable (so that
$\hat{\pi}$ is square integrable, too) we derive a general
inversion formula for $\cR$ of the form 
\begin{equation}\label{eq:intro}
 f =  \int_G \chi(g) \scal{\cR f}{\hat{\pi}(g)\Psi} \,     \pi(g)\psi\ {\rm d}\mu(g),
\end{equation}
where $\chi$ is a character of $G$ and $\psi\in L^2(X,{\rm d}x)$ and $\Psi\in L^2(\Xi,{\rm d}\xi)$ are
suitable mother wavelets and the Haar integral is weakly convergent, see
Theorem~\ref{generalinversionformula}. We
stress that the coefficients $\scal{\cR f}{\hat{\pi}(g)\Psi}$ depend on $f$
only through its Radon transform $\cR f$, so that  the above equation allows to
reconstruct an unknown signal from its Radon transform by
computing the family of coefficients $\{\scal{\cR f}{\hat{\pi}(g)\Psi}\}_{g\in G} $. As it is clear from \eqref{eq:intro},  $\hat\pi$ is used as an ``analysis'' transform applied to $\cR f$ and $\pi$ as a ``synthesis'' transform to reconstruct $f$.  This kind of reconstruction
formulae is already known for the classical Radon trasform where $G$ is
the affine group of $\R^d$ associated with the multi-dimensional
wavelets, \cite{holschneider91,walnut-1993, olson-destefano-1994,
  berenstein-walnut-1994,madych-1999}, and for the affine Radon
transform where $G$ is the shearlet group \cite{bardemadeviodo}.

We illustrate the construction and the result with the examples where $G$ is either the similitude group of the plane
(with two different choices of $\Xi$) or the standard shearlet group \cite{lawakuwe05}, but
other cases could also be covered, such as the generalized shearlet
dilation groups \cite{futo16,albertietal17}. We believe that our
contribution may  be further substantiated with several other
examples and deepened in several directions.  In particular, it would be interesting to consider more general groups $G$ and to relax the assumption of irreducibility of the representations. This would allow us to include many other examples, such as the class of groups studied in \cite{alberti13,alberti14}.

\vskip0.2truecm
For clarity, we list the main assumptions that are made along the
way.  We consider:
\begin{itemize}
\item  a Lie group $G=\R^{d}\rtimes K$, where $K$ is a
  closed subgroup of ${\rm GL}(d,\R)$; 
\item the space $X=\R^d$, regarded as a transitive
  $G$-space  with respect to the canonical action of $G$ (denoted by $(b,k)[x]=b+kx$ for $(b,k)\in G$ and $x\in X$), equipped with the Lebesgue measure $\D x$; 
\item a smooth transitive $G$-space $\Xi$ (we denote the action of $G$ on $\Xi$ by $g.\xi$) with
  an origin  $\xi_0\in\Xi$ and $H$ the isotropy  at $\xi_0$; 
  \end{itemize}
and we assume that the following conditions hold true: 
\begin{enumerate}[label=(A\arabic*)]
\item\label{ass:dxdXI}   the space $\Xi$ carries a relatively
    $G$-invariant  measure ${\rm d}\xi$;
\item\label{ass:dm_0} denoting  the origin of $X$ by $x_0=0$, the $H$-transitive space
  \[
\hat\xi_0=H[x_0] \subset X
    \]
carries a relatively $H$-invariant measure $m_0$ with character $\gamma$;
\item\label{ass:sigma} there exists a Borel section $\sigma\colon\Xi\to G$ such that 
\[
(g,\xi)\mapsto\gamma\bigl(\sigma(\xi)^{-1}g\sigma(g^{-1}.\,\xi)   \bigr)
\]
extends to a positive character of $G$ independent of $\xi$;
\item \label{ass:dualpair}   the pair $(X,\Xi)$  is a dual pair in the
  sense of Helgason under the canonical isomorphism $X=G/K$ and $\Xi=G/H$;
\item\label{ass:irreducible}   the quasi-regular representation $\pi$ of $G$
    acting on $L^2(X,{\rm d}x)$  is irreducible and square-integrable;
\item\label{ass:irreducible1}  the quasi-regular representation $\hat{\pi}$
  of $G$ acting on   $L^2(\Xi,{\rm d}\xi)$ is irreducible;
\item\label{ass:A} there exists a non-trivial $\pi$-invariant
  subspace  $\cA \subseteq L^2(X, {\rm d}x)$ such that
\begin{align*}
  & f(\sigma(\xi)[\cdot])\in L^1(\hat{\xi}_0,m_0)  \quad 
  \text{for almost all }\xi\in\Xi, \\
 & \mathcal{R}f :=\int_{\hat{\xi}_0}  f(\sigma(\cdot)[x])  {\rm d}m_0(x)
 \in L^2(\Xi, {\rm d}\xi),
\end{align*}
for all
  $f\in\mathcal A$,  and the map $f\mapsto \mathcal{R}f$ is  a closable operator from
$\mathcal A$ to  $L^2(\Xi, {\rm d}\xi)$.
\end{enumerate}
We note that the condition that $(X,\Xi)$ is a dual pair (assumption~\ref{ass:dualpair}) will not
play any explicit role, as it is implied by \ref{ass:irreducible} and  \ref{ass:irreducible1} (see Remark~\ref{rem:dual} below). The assumption that $\pi$ is square-integrable is
needed only in Section~\ref{inversionsection}, whereas the condition
that $\hat{\pi}$ is irreducible is not needed  until Theorem~\ref{ernestoduflomoore}.

\section{Preliminaries}
\subsection{Notation}
We briefly introduce the notation. We set
$\R^{\times}=\R\setminus\{0\}$ and $\R^{+}=(0,+\infty)$. The Euclidean norm of a vector
$v\in\R^d$ is denoted by $|v|$ and its scalar product with $w\in\R^d$ 
by $v\cdot w$. For any $p\in[1,+\infty]$ we denote
by $L^p(\R^d)$ the Banach space of functions $f\colon\R^d\rightarrow\C$ that are $p$-integrable with respect to the Lebesgue measure $\D x$
and, if $p=2$, the corresponding scalar product and norm are
$\langle\cdot,\cdot\rangle$ and $\|\cdot\|$, respectively. If $E$
  is a Borel subset of $\R^d$, $|E|$ also denotes its Lebesgue measure.
The Fourier trasform is denoted by $\mathcal F$ both on
$L^2(\R^d)$ and on  $L^1(\R^d)$, where it is 
defined by
\[
\mathcal F f({\xi}\,)= \int_{\R^d} f(x) \E^{-2\pi i\,
  {\xi}\cdot x } \D{x},\qquad f\in L^1(\R^d).
\] 
If $G$ is a Lie group, we denote by $L^2(G)$ the Hilbert
space of square-integrable functions with respect to a left Haar
measure on $G$. If $X$ is a smooth
transitive $G$-space with origin $x_0$, denoted  by $g[x]$ the action of $G$ on $X$, a Borel measure $\mu$
 of $X$ is relatively invariant if there exists a positive  character  $\alpha$ of $G$ such
that for any measurable set $E\subset  X$ and $g\in G$ it holds
$\mu(g[E])= \alpha(g)\mu(E)$. Furthermore,  a Borel section is a measurable map
    $s\colon X\to G$ satisfying  $s(x)[x_0]=x$
and $s(x_0)=e$, with $e$ the neutral element of $G$; a Borel section always
exists since $G$ is second countable \cite[Theorem~5.11]{varadarajan85}. 
We denote the (real) general linear group of size
$d\times d$ by ${\rm GL}(d,\R)$.

\subsection{Dual homogeneous spaces of semidirect products}
In this section we recall the basic construct due to Helgason
\cite{helgason99} adapted to the context in  
 which we are interested, namely when the group $G$ is the semidirect
 product of the Euclidean space $\R^d$ with a closed subgroup
 $K$ of ${\rm GL}(d,\R)$.
 This structure  is enjoyed by several
 groups of interest in applications, such as the similitude group
 studied by Murenzi \cite{anmu96}, and the generalized shearlet
 dilation groups introduced by F$\"u$hr in \cite{fu98,futo16} for the
 purpose of generalizing the standard shearlet group introduced in
 \cite{lawakuwe05,dastte10}.  Whenever possible, we keep the notation as
 in~\cite{helgason99}.

We recall that $G=\R^d\rtimes K$ is the manifold
  $\R^{d}\times K$ endowed with the group operation
\[ 
(b_1,k_1)(b_2,k_2)=(b_1+k_1b_2,k_1k_2),\qquad b_1, b_2\in\R^d,\;k_1,k_2\in K,
\]
where  $kb$ is the natural linear action of  the matrix $k$ on the
column vector $b$, so that $G$ is a Lie group.
The inverse of an element in $G$ is given by
$(b,k)^{-1}=(-k^{-1}b,k^{-1})$.
A left Haar measure of $G$ is 
\begin{equation}\label{eq:haar}
{\rm d}\mu(b,k)=|\det k|^{-1}{\rm d}b{\rm d}k,
\end{equation}
where ${\rm d}b$ is the Lebesgue measure of $\R^d$
and ${\rm d}k$ is a left Haar measure on $K$.

The first transitive space we consider is  $X=\R^d$, regarded  as
smooth $G$-space with respect to the canonical action
\[
(b,k)[x]=b+kx, \qquad  (b,k)\in G,\; x\in X.
\]
The action is clearly transitive, the isotropy at the origin $x_0=0$ is the subgroup $\{(0,k):k\in K\}$
which we identify with $K$, so that $X\simeq G/K$. Furthermore, the map 
  \[
s\colon X\to G, \qquad s(x)=(x,\operatorname{I}_d),
  \]
is a Borel section and  the Lebesgue
measure ${\rm d}x$ on $X$ is a relatively $G$-invariant measure,
since for any measurable set $E\subseteq\R^d$ we have
\[|(b,k)[E]|=|b+kE|=|kE|=|\det k||E| .\]

We consider a second smooth transitive  $G$-space
  $\Xi$, whose action  is denoted 
  by
  \[ G\times\Xi\ni (g,\xi) \mapsto g.\,\xi\in\Xi. \]
  By Assumption~\ref{ass:dxdXI},  $\Xi$ admits a relatively invariant
measure ${\rm d}\xi$, which may be expressed by  the equality 
\begin{equation}
  \label{eq:1}
  \int_{\Xi} f( g^{-1}.\xi) \, {\rm d}\xi =\beta(g) \int_{\Xi} f( \xi) \, {\rm d}\xi,\qquad f\in L^1(\Xi, {\rm d}\xi),\;g\in G,
\end{equation}
where $\beta\colon G\to (0,+\infty)$ is a positive character of~$G$.

We fix an origin $\xi_0\in\Xi$, we denote  the isotropy at~$\xi_0$ by
$H$  and we put 
\[
\check{x}_0=K.\,\xi_0\subset\Xi,\qquad \hat{\xi}_0=H[x_0]\subset X.
\]
By definition, $\check{x}_0$ and $\hat{\xi}_0$ are $K$ and $H$
transitive spaces, respectively. In order to define the Radon transform we will make use of
  Assumption~\ref{ass:dm_0}, 
namely that  $\hat{\xi}_0$
carries a relatively $H$-invariant Radon measure~${\rm d}m_0$, that is
\begin{equation}
  \label{mzero}
  \int_{\hat\xi_0} f( h^{-1}[x]) \, {\rm d}m_0(x) =\gamma(h) \int_{\hat\xi_0}  f( x) \, {\rm d}m_0(x),\qquad f\in L^1(\hat{\xi_0},\D m_0),\;h\in H,
\end{equation}
where
$\gamma\colon H\to (0,+\infty)$ is a positive character of~$H$.  This
is a weaker assumption than in Helgason's approach,
in which $\hat{\xi}_0$ is assumed to admit a {\it bona fide} invariant measure for the $H$-action.

We fix a Borel section $\sigma\colon\Xi\to G$
 such that~\ref{ass:sigma} holds true. 
With an equivalent approach to that of Helgason's, we define the
sets
\begin{equation}\label{spreads}
\hat{\xi}=\sigma(\xi)[\hat{\xi}_0]\subset X,\qquad
\check{x}=s(x).\,\check{x}_0\subset \Xi,
\end{equation}
which are closed subsets  by~\cite[Lemma~1.1]{helgason99}.

Assumption~\ref{ass:dualpair} states that  $(X,\Xi)$ is a dual
  pair, which means that the maps $x\mapsto\check{x}$ and 
  $\xi\mapsto\hat{\xi}$ are both injective. This assumption is called
  transversality, see  
  \cite[Lemma~1.3]{helgason99} about an equivalent characterization.  Apart from
the cases considered below, the reader may consult~\cite{helgason99}
for  numerous examples of dual pairs $(X,\Xi)$.  
\begin{oss}\label{rem:dual}
As mentioned in the introduction, transversality is in fact implied by assumptions \ref{ass:irreducible} and \ref{ass:irreducible1} (we leave this investigation for more general groups and spaces to future work) and will never be used in the arguments below.
\end{oss}
The following example shows that the (classical) Radon transform
  can be obtained in this framework.
   Two other examples are illustrated in Section \ref{examples}.
\begin{ex}\label{ex:SIM}
The (connected component of the identity of the) similitude group $SIM(2)$ of the plane is $\R^2\rtimes K$, with $K=\{R_\phi A_a\in \text{GL}(2,\R):\phi\in [0,2\pi),\, a\in\R^+\}$ where \[
R_{\phi}=\left[\begin{matrix}\cos{\phi} & -\sin{\phi} \\ \sin{\phi} & \cos{\phi}\end{matrix}\right],\qquad A_a=\left[\begin{matrix}a & 0 \\ 0 & a\end{matrix}\right].
\]
By the identification $K\simeq [0,2\pi)\times\R^{+}$, we write $(b,\phi,a)$ for the elements in $SIM(2)$. With this identification the group law becomes
\[
(b,\phi,a)(b',\phi',a')=(b+R_{\phi}A_ab', \phi+\phi'\ \text{mod}\ 2\pi,aa'),
\]
and the inverse of $(b,\phi,a)$ is given by
 \begin{equation}\label{eq:inverseSIM}
(b,\phi,a)^{-1}=(-A_a^{-1}R_{\phi}^{-1}b,-\phi\ \text{mod}\ 2\pi, a^{-1}).
\end{equation}
By \eqref{eq:haar}, a left Haar measure of $SIM(2)$ is 
\begin{equation}\label{eq:haarSIM}
\D\mu(b,\phi,a)=a^{-3}\D b\D \phi\D a,
\end{equation}
where $\D b$, $\D \phi$ and $\D a$ are the Lebesgue measures on
  $\R^2$, $[0,2\pi)$ and $\R_{+}$, respectively.

It remains to choose the space $\Xi$ and the corresponding subgroup $H$ of $SIM(2)$. The group $SIM(2)$ acts transitively on $\Xi=[0,\pi)\times\R$ by 
\[
(b,\phi,a).(\theta,t)=(\theta+\phi\ \text{mod}\ \pi,a(t+\nt(\theta)\cdot A_a^{-1}R_{\phi}^{-1}b)),
\]
where $\nt(\theta)={^t(\cos{\theta},\sin{\theta})}$, or equivalently 
\[
(b,\phi,a)^{-1}.(\theta,t)=\left(\theta-\phi\ \text{mod}\
  \pi,\frac{t-\nt(\theta)\cdot b}{a}\right).
\]
The isotropy at $\xi_0=(0,0)$ is 
\[
H=\{((0,b_2),\phi,a):b_2\in\R,\phi\in\{0,\pi\},a\in\R^+\}.
\]
Thus, $[0,\pi)\times\R=SIM(2)/H$.
An immediate calculation gives
\[
\int_\Xi f\left((b,\phi,a)^{-1}.(\theta,t)\right)\D\theta\D t=a \int_\Xi f\left(\theta,t\right)\D\theta\D t,\qquad f\in L^1(\Xi,\D\theta\D t),
\]
namely, \eqref{eq:1} is satisfied with the character $\beta(b,\phi,a)=a$. Thus, the Lebesgue measure $d\xi=\D\theta\D t$ is a relatively invariant measure on $\Xi$.

Consider now the sections $s\colon\R^2\to SIM(2)$ and $\sigma\colon [0,\pi)\times\R\to SIM(2)$ defined by
\[
s(x)=(x,0,1),\qquad
\sigma(\theta,t)=(t\, \nt(\theta),\theta,1).
\]
It is easy to verify by direct computation that
\begin{align*}
\hat{\xi}_0&=H[x_0]=\{(0, b_2): b_2\in\R\}\simeq\R,\\
 \check{x}_0&=K.\xi_0=\{(\theta,0):\theta\in[0,\pi)\}\simeq [0,\pi).
\end{align*}
It is immediate to see that the Lebesgue measure $\D b_2$ on
$\hat{\xi}_0$ is a relatively $H$-invariant measure
with character $\gamma((0, b_2),\phi,a)=a$. Further, we have that
\[
\widehat{(\theta,t)}=\sigma(\theta,t)[\hat{\xi}_0]=\{x\in\R^2:x\cdot \nt(\theta)=t\},
\]
which is the set of all points laying on the line of equation $x\cdot \nt(\theta)=t$ and 
\[
\check{x}=s(x).\check{x}_0=\{(\theta,t)\in [0,\pi)\times\R:t-\nt(\theta)\cdot x=0\},
\]
which parametrizes the set of all lines passing through the point $x$.

It is easy to verify that $X=\R^2$ and $\Xi=[0,\pi)\times\R$ are homogeneous spaces in duality. Indeed, the map $x\mapsto\check{x}$ which sends a point to the set of all lines passing through that point and the map $(\theta,t)\mapsto\widehat{(\theta,t)}$ which sends a line to the set of points laying on that line are both injective. 
\end{ex}

\subsection{The representations}
The group $G$ acts unitarily on  $L^2(X,{\rm d}x)$ via the quasi-regular representation defined by
\[
\pi(g)f(x)=|\det(k)|^{-1/2}f(k^{-1}(x-b)),\qquad g=(b,k).
\]
By
  Assumption~\ref{ass:irreducible},  $\pi$ is
irreducible and  square-integrable. We stress that this latter
condition is needed only in Section~\ref{inversionsection}.

The group $G$  acts also on $L^2(\Xi,{\rm d}\xi)$ via the quasi-regular representation 
\[
\hat{\pi}(g)F(\xi)=\beta(g)^{-1/2}F(g^{-1}.\,\xi),
\]
where $\beta(g)$  is defined in~\eqref{eq:1}. The representation
$\hat{\pi}$ is irreducible by~Assumption~\ref{ass:irreducible};
however this condition is not needed  until Theorem~\ref{ernestoduflomoore}.

\begin{exSIM}
The group $SIM(2)$ acts on $L^2(\R^2)$ by means of the unitary
irreducible representation
$\pi$ defined by
\begin{equation}\label{simrepresentation}
\pi(b,\phi,a)f(x)=a^{-1}f(A_a^{-1}R_{\phi}^{-1}(x-b)),
\end{equation}
or, equivalently, in the frequency domain
\begin{equation}\label{simrepresentationfrequency}
\mathcal{F}[\pi(b,\phi,a)f](\xi)=ae^{-2\pi ib\cdot\xi}\mathcal{F}f(A_aR_{\phi}^{-1}\xi).
\end{equation}
As above, we consider on $\Xi=[0,\pi)\times\R$ the measure ${\rm d}\xi={\rm d}\theta{\rm d}t$, where ${\rm d}\theta$ and ${\rm d}t$ are the Lebesgue measures on $[0,\pi)$ and $\R$ respectively. Then, since $\beta(b,\phi,a)=a$, $G$ acts on $L^2([0,\pi)\times\R,{\rm d}\theta{\rm d}t)$ by means of the quasi-regular representation $\hat{\pi}$ defined by
\begin{equation}\label{simrepresentationhat}
\hat{\pi}(b,\phi,a)F(\theta,t)=a^{-\frac{1}{2}}F\left(\theta-\phi\ \text{mod}\ \pi,\frac{t-\nt(\theta)\cdot b}{a}\right),
\end{equation}
which is irreducible, too. 
\end{exSIM}

\subsection{The Radon transform}

Following Helgason's theory and using the Borel section $\sigma$ in order to push-forward the measure ${\rm d}m_0$ (see \eqref{mzero}) to the manifolds $\hat{\xi}$ given in  \eqref{spreads}, we define the Radon transform of $f$ as the map $\mathcal{R}f\colon\Xi\to\C$ given  by
\begin{equation}\label{generalradon}
\mathcal{R}f(\xi)=\int_{\hat{\xi}}\,f(x) {\rm d}m_\xi(x):=\int_{\hat{\xi}_0}f(\sigma(\xi)[x]){\rm
  d}m_0(x) .
\end{equation}
Note that this depends intrinsically on the choices of $\D m_0$ and $\sigma$, and not only on the subset of integration  $\hat\xi$.

Assumption~\ref{ass:A} states that there exists a non-trivial $\pi$-invariant
  subspace $\mathcal A$ of $L^2(X, {\mathrm d}x)$ such that  $\mathcal{R}f$ is well defined for all
  $f\in\mathcal A$ and the Radon transform $\mathcal{R}$ is a closable operator from
  $\mathcal A$ into $L^2(\Xi,{\mathrm d}\xi)$.  We denote its closure by
  $\overline{\mathcal{R}}$. Clearly,  $\mathcal{R}f$ may be well defined
  for a larger subset of functions than   $\mathcal A$, however we will  show that
  $\overline{\mathcal{R}}$  is independent of the choice of $\mathcal A$.  


\begin{exSIM}
We compute by
\eqref{generalradon} the Radon transform between the homogeneous
spaces in duality $\R^2$ and $[0,\pi)\times\R$ and we obtain 
\begin{equation}\label{radonpolar}
\mathcal{R}^{\text{pol}}f(\theta,t)=\int_{\R}f(t\cos{\theta}-y\sin{\theta},t\sin{\theta}+y\cos{\theta}){\rm d}y,
\end{equation}
which is the so-called polar Radon transform.

Next we choose the domain $\cA$ of $\cR^{\rm pol}$, which will be called $\cA^{\rm pol}$. This requires to recall one of the
fundamental results in Radon theory, the so-called Fourier slice
theorem \cite{helgason99}, which relates the Radon transform with the Fourier
transform. For $f\in L^1(\R^2)$, the integral
  \eqref{radonpolar} converges for almost all
  $(\theta,t)\in[0,\pi)\times\R$ by Fubini's theorem and
  \begin{equation}
    \label{eq:FST}
(I\otimes\mathcal F )(\mathcal{R}^{{\rm
    pol}}f)(\theta,\tau)
=\mathcal{F}f(\tau\nt(\theta))  
  \end{equation}
 for  every $(\theta,\tau)\in [0,\pi)\times \R$, where $I$ is
the identity operator. Hence, we have
\begin{align*}
\int_{[0,\pi)\times\R}|\cR^{\rm pol}f(\theta,t)|^2\ {\rm d}\theta{\rm d}t&=\int_{[0,\pi)}\int_{\R}| (I\otimes\mathcal F ) (\mathcal{R}^{{\rm pol}}f)(\theta,\tau)|^2\ {\rm d}\tau{\rm d}\theta\\
&=\int_{[0,\pi)\times\R}|\cF f(\tau \nt(\theta))|^2{\rm d}\theta{\rm d}\tau\\
&=\int_{\R^2}\frac{|\cF f(\xi_1,\xi_2)|^2}{\sqrt{\xi_1^2+\xi_2^2}}{\rm d}\xi_1{\rm d}\xi_2.
\end{align*}
We are thus  led to choose
\[
\cA^{\text{pol}}=\{f\in L^1(\R^2)\cap L^2(\R^2):\int_{\R^2}\frac{|\cF f(\xi_1,\xi_2)|^2}{\sqrt{\xi_1^2+\xi_2^2}}{\rm d}\xi_1{\rm d}\xi_2<+\infty\},
\]
which by \eqref{simrepresentationfrequency} is $\pi$-invariant and by definition, $\cR^{\rm pol}f\in
L^2([0,\pi)\times\R)$ for all $f\in\cA^{\text{pol}}$. 

Now, we show that $\cR^{\rm
  pol}$ restricted to $\cA^{\text{pol}}$ is closable. Suppose that $(f_n)_n\subset\cA^{\rm pol}$ is a
sequence such that $f_n\to f$ in $L^2(\R^2)$ and $\cR^{\rm pol}f_n\to
g$ in $L^2([0,\pi)\times\R)$. Since $I\otimes\mathcal{F}$ is unitary
from $L^2([0,\pi)\times\R)$ onto $L^2([0,\pi)\times\R)$, we have that $(I\otimes\mathcal{F})\cR^{\rm
  pol}f_n\to(I\otimes\mathcal{F})g$ in $L^2([0,\pi)\times\R)$. Since
$f_n\in\cA^{\rm pol}$, by~\eqref{eq:FST}, for  every
$(\theta,\tau)\in[0,\pi)\times\R$ 
\begin{align*}
(I\otimes\mathcal{F})\cR^{\rm pol}f_n(\theta,\tau)&=\mathcal{F}f_n(\tau \nt(\theta)).
\end{align*}
Hence, passing to a subsequence if necessary,
\begin{align*}
\mathcal{F}f_n(\tau \nt(\theta))\to (I\otimes\mathcal{F})g(\theta,\tau)
\end{align*}
for almost every $(\theta,\tau)\in[0,\pi)\times\R$. Therefore, for almost every $(\theta,\tau)\in[0,\pi)\times\R$
\begin{align*}
(I\otimes\mathcal{F})g(\theta,\tau)=\lim_{n\to+\infty}\mathcal{F}f_n(\tau \nt(\theta))=\mathcal{F}f(\tau \nt(\theta)),
\end{align*}
where the last equality holds true using a subsequence if necessary. Therefore, if $(h_n)_n\in\cA^{\rm pol}$ is another sequence such that $h_n\to f$ in $L^2(\R^2)$ and $\cR^{\rm pol}h_n\to h$ in $L^2([0,\pi)\times\R)$, then, for almost every $(\theta,\tau)\in[0,\pi)\times\R$
\begin{align*}
(I\otimes\mathcal{F})h(\theta,\tau)=\mathcal{F}f(\tau \nt(\theta)).
\end{align*}
Therefore,
\[
(I\otimes\mathcal{F})g(\theta,\tau)=(I\otimes\mathcal{F})h(\theta,\tau)
\]
for almost every $(\theta,\tau)\in[0,\pi)\times\R$. Then $\lim_{n\to+\infty}\cR^{\rm pol}f_n=\lim_{n\to+\infty}\cR^{\rm pol}h_n$, and $\cR^{\rm pol}$ is closable.
\end{exSIM}

\section{The Unitarization Theorem}\label{mainsection}

Our construction is based on the following lemma, which shows that the
Radon transform intertwines the representations $\pi$ and $\hat\pi$ up
to a positive character of $G$.
For the classical Radon transform considered in Example~\ref{ex:SIM},
this result is a direct consequence of the behavior of $\cR^{\rm pol}$ under linear actions \cite[Chapter 2]{raka96}.
\begin{lem}\label{lem:R} 
 The Radon transform $\cR$ restricted to $\cA$
is a densely defined operator from $\cA$ into $L^2(\Xi,{\rm d}\xi)$
satisfying
\begin{equation}
\label{intertwining}
\mathcal{R}\pi(g) =\chi(g)^{-1}\hat{\pi}(g) \mathcal{R} , 
\end{equation}
for all  $g\in G$, where
\begin{equation}\label{chi}
\chi(g)=\beta(g)^{-1/2}|\det{(k)}|^{1/2}\gamma(g\sigma(g^{-1}.\xi_0))^{-1},\qquad g=(b,k)\in G.
\end{equation}
\end{lem}
With a slight abuse of notation, $\cR$ denotes both the Radon
  transform defined by~\eqref{generalradon} and its restriction to $\cA$.
\begin{proof}
By Assumption~\ref{ass:A}, $\cR$ is a well-defined operator from
  $\cA$ into $L^2(\Xi,{\rm d}\xi)$. 
We now  prove~\eqref{intertwining}.  By the $\pi$-invariance of $\mathcal A$, for $f\in \mathcal A$ and
$g=(b,k)\in G$ we have 
\begin{align*}
(\mathcal{R}\pi(b,k)f)(\xi)&=|\det{(k)}|^{-1/2}\int_{\hat{\xi}_0}f((b,k)^{-1}\sigma(\xi)[x]){\rm d}m_0(x)\\
&=|\det{(k)}|^{-1/2}\int_{\hat{\xi}_0}f(\sigma((b,k)^{-1}.\,\xi)\sigma((b,k)^{-1}.\,\xi)^{-1}(b,k)^{-1}\sigma(\xi)[x]){\rm d}m_0(x)\\
&=|\det{(k)}|^{-1/2}\int_{\hat{\xi}_0}f(\sigma((b,k)^{-1}.\,\xi)m((b,k),\xi)^{-1}[x]){\rm d}m_0(x),
\end{align*}
where
$m((b,k),\xi)^{-1}:=\sigma((b,k)^{-1}.\,\xi)^{-1}(b,k)^{-1}\sigma(\xi)$.
It is known that 
for any $g\in G$ and any $\xi\in\Xi$
\begin{equation}\label{Borel}
m(g,\xi)=\sigma(\xi) ^{-1}g \sigma(g^{-1}.\,\xi)\in H.
\end{equation} 
We show this  property for the reader's convenience. Indeed
\[
\sigma(\xi) ^{-1}g \sigma(g^{-1}.\,\xi).\xi_0= \sigma(\xi) ^{-1}g.( g^{-1}.\,\xi)=\sigma(\xi) ^{-1}.\xi=\xi_0,
  \]
so that $ m(g,\xi)\in H$.  Thus, using \eqref{mzero} we obtain
\begin{align*}
(\mathcal{R}\pi(b,k)f)(\xi)&=|\det{(k)}|^{-1/2}\gamma(m((b,k),\xi))\int_{\hat{\xi}_0}f(\sigma((b,k)^{-1}.\xi)[x]){\rm d}m_0(x).
\end{align*}
Then,
\begin{align*}
(\mathcal{R}\pi(b,k)f)(\xi)&=|\det{(k)}|^{-1/2}\gamma(m((b,k),\xi))(\mathcal{R}f)((b,k)^{-1}.\xi)\\
&=\beta(b,k)^{1/2}|\det{(k)}|^{-1/2}\gamma(m((b,k),\xi))\hat{\pi}(b,k)\mathcal{R}f(\xi).
\end{align*}
Thanks to assumption \ref{ass:sigma}, $(g,\xi)\mapsto\gamma(m(g,\xi))$ extends to a
positive character of $G$  independent of $\xi$. In particular, $\gamma(m(g,\xi))=\gamma(m(g,\xi_0))=\gamma(g\sigma(g^{-1}.\xi_0))$, and \eqref{intertwining} follows.

We finally prove that $\cA$ is dense. By assumption~\ref{ass:A},
the domain of $\mathcal{R}$ is $\pi$ invariant,  so that
$\pi(g)\overline{\mathcal{A}}\subset\overline{\mathcal{A}}$ for every
$g\in G$. Since $\mathcal A\neq\{0\}$ and $\pi$ is irreducible,
then  $\overline{\mathcal{A}}=L^2(X,{\rm d}x)$. 
\end{proof}

Observe that, if $\gamma$ extends to a positive character of $G$, then 
\[
\gamma(m(g,\xi))=\gamma(\sigma(\xi))^{-1}\gamma(g)\gamma(\sigma(g^{-1}.\,\xi))
\]
and the independence of $\xi$ is implied by the stronger condition
\[
\gamma(\sigma(g^{-1}.\,\xi))=\gamma(\sigma(\xi)),
\]
that must be satisfied for all $g\in G$ and $\xi\in\Xi$. This is equivalent to requiring that $\gamma(\sigma(\xi))=1$ for all $\xi\in\Xi$, which is true in  all our examples.

\begin{exSIM}
By $\eqref{intertwining}$ and $\eqref{chi}$ we have that 
\begin{align}\label{polarintertwining}
\mathcal{R}^{\text{pol}}\pi(b,\phi,a) &=\chi(b,\phi,a)^{-1}\hat{\pi}(b,\phi,a)\mathcal{R}^{\text{pol}},
\end{align}
where $\chi(b,\phi,a)=a^{-1/2}$ is a character of $G$, since $\beta(b,\phi,a)=\gamma(b,\phi,a)=a$ and $\det(R_\phi A_a)=a^2$. 
\end{exSIM}

The following result is the key of our construction. Recall that by Assumption~\ref{ass:A} the operator $\mathcal{R}\colon\cA\subseteq L^2(X,{\rm d}x)\to
  L^2(\Xi,{\rm d}\xi)$ is closable.  
  \begin{lem}\label{lem:Rbar}
The  closure $\overline{\mathcal{R}}$ of the Radon transform $\cR$ is a densely defined operator satisfying
\begin{equation}
  \label{character}
\overline{\mathcal{R}}\pi(g) =\chi(g)^{-1}\hat{\pi}(g) \overline{\mathcal{R}}, 
\end{equation}
for all $g\in G$, where $\chi$ is given by \eqref{chi}.
\end{lem}

\begin{proof}
Since $\mathcal{R}$ is densely defined by Lemma~\ref{lem:R},
its closure $\overline{\mathcal{R}}$ is a densely defined
operator from $L^2(X,{\rm d}x)$ to  $L^2(\Xi,{\rm d}\xi)$.
Now, we prove that  $\overline{\mathcal{R}}$
satisfies~\eqref{character}.
We start by proving that the domain of $\overline{\mathcal{R}}$ is
$\pi$-invariant.  Given  $f\in
\operatorname{dom}(\overline{\mathcal{R}})$,  let
$(f_n)_n\in\cA$  be a sequence such that $f_n\to f$ in $L^2(X,{\rm d}x)$ and  $ \mathcal{R}f_n\to\overline{\mathcal{R}}f$ in $L^2(\Xi,{\rm
  d}\xi)$. In order to prove that $\pi(g)f \in
\operatorname{dom}(\overline{\mathcal{R}})$ it is enough to show that
the sequence $(\mathcal{R}\pi(g)f_n)_n$ converges in $L^2(\Xi,{\rm
  d}\xi)$: indeed $\pi(g)f_n\in\cA$ for any $n\in\N$ because
$\mathcal{A}$ is $\pi$-invariant and $\pi(g)f_n\to\pi(g)f$ in
$L^2(X,{\rm d}x)$ because $\pi(g)$ is a unitary operator for any $g\in
G$. But then, by  property \eqref{intertwining} we have
that 
\begin{equation}\label{intertwiningclosure}
\begin{split}
\lim_{n\to+\infty}\mathcal{R}\pi(g)f_n&=\chi(g)^{-1}\lim_{n\to+\infty}\hat{\pi}(g)\mathcal{R}f_n\\
&=\chi(g)^{-1}\hat{\pi}(g)\lim_{n\to+\infty}\mathcal{R}f_n\\
&=\chi(g)^{-1}\hat{\pi}(g)\overline{\mathcal{R}}f,
\end{split}
\end{equation}
so that $(\mathcal{R}\pi(g)f_n)_n$ converges in $L^2(\Xi,{\rm d}\xi)$. Then, $\pi(g)f \in \operatorname{dom}(\overline{\mathcal{R}})$ and by definition $\mathcal{R}\pi(g)f_n\to\overline{\mathcal{R}}\pi(g)f$. Therefore, by \eqref{intertwiningclosure}, $\overline{\mathcal{R}}$ satisfies \eqref{character}.
\end{proof}

Our main theorem  is based on the  following classical result due to Duflo and Moore  \cite{dumo76}. According to  \cite{dumo76}, a densely defined closed operator
  $T$ from a Hilbert space $\mathcal{H}$ to another Hilbert space
  $\hat{\mathcal{H}}$ is called semi-invariant 
  with weight $\zeta$ if it satisfies  
\begin{align}\label{semiinvariant}
\hat{\pi}(g)T\pi(g)^{-1}=\zeta(g)T, \qquad g\in G,
\end{align}
where  $\zeta$ is a  character of $G$ and $\pi$ and $\hat{\pi}$ are unitary
representations of $G$ acting on $\cH$ and $\hat{\cH}$, respectively.
\begin{teo}[{\cite[Theorem 1]{dumo76}}]\label{duflomoore}
With the above notation, assume that $\pi$ is irreducible. Let $T$ be a
densely defined closed nonzero operator from $\mathcal{H}$ to
$\hat{\mathcal{H}}$, semi-invariant with weight $\zeta$.
\begin{enumerate}[label=(\roman*)]
\item Suppose that $\pi=\hat{\pi}$. Let $T'$ be another densely
  defined closed operator from $\mathcal{H}$ to $\mathcal{H}$,
  semi-invariant with weight $\zeta$. Then $T'$ is proportional to
  $T$. 
\item Let $T=\cQ|T|$ be the polar decomposition of $T$. Then
  $|T|$ is a positive selfadjoint operator in $\mathcal{H}$
  semi-invariant with weight $|\zeta|$, and $\cQ$ is a partial isometry
  of $\mathcal{H}$ into $\hat{\mathcal{H}}$, semi-invariant with weight
  $\zeta/|\zeta|$. 
\end{enumerate}
\end{teo}

By observing that Lemma \ref{lem:Rbar} shows that the Radon
  transform $\overline{\mathcal R}$  is a semi-invariant operator with
  weight given by~\eqref{chi}, we  are finally in a position
to state and prove our main 
result. We stress that its proof  does not use the transversality condition on $(X,\Xi)$ (assumption \ref{ass:dualpair}, cfr.\ Remark~\ref{rem:dual}) and the square-integrability of $\pi$; the irreducibility of $\hat{\pi}$ 
  is only needed in the last claim of the theorem.
  \begin{teo}\label{ernestoduflomoore}
There exists a unique positive self-adjoint operator
  \[ \mathcal{I}\colon \operatorname{dom}(\mathcal{I}) \supseteq
    \operatorname{Im}\overline{\mathcal{R}}\to L^2(\Xi,{\rm d}\xi), 
  \] 
semi-invariant with weight  $\zeta=\chi^{-1}$ with the property that 
 the composite operator $\cI \overline{\mathcal{R}}$ extends to
    an isometry
$\cQ\colon L^2(X,{\rm d}x)\to L^2(\Xi,{\rm d}\xi)$
intertwining $\pi$ and  $\hat{\pi}$, namely
   \begin{equation}\label{intertwiningU}
\hat{\pi}(g)\cQ\pi(g)^{-1}=\cQ,
\qquad g\in G.
\end{equation}
Furthermore, if $\hat{\pi}$ is irreducible, then $\cQ$ is a unitary operator
and $\pi$ and  $\hat{\pi}$ are equivalent representations.
\end{teo}
The above result is a generalization of Helgason's theorem  on the
unitarization of the classical Radon transform, \cite[Theorem~4.1]{helgason99},
because by definition of extension it holds that 
    \begin{equation}
\cI \mathcal{R} f=\cQ f,\qquad f\in \cA .\label{eq:5}
\end{equation}

\begin{proof}
The unitarization  of $\cR$ is based on 
the polar decomposition $\overline{\mathcal{R}}=\cQ
|\overline{\mathcal{R}}|$  of $\overline{\mathcal{R}}$. By  Lemma~\ref{lem:Rbar} and
Theorem~\ref{duflomoore}, item~(ii),
$|\overline{\mathcal{R}}|\colon\operatorname{dom}(\overline{\mathcal{R}})\to
L^2(X,{\rm d}x)$ is a positive self-adjoint operator semi-invariant
with weight $|\chi|=\chi$, where $\chi$ is defined by \eqref{chi},  i.e.
\begin{equation}\label{intertwiningH}
\pi(g)|\overline{\mathcal{R}}|\pi(g)^{-1}=\chi(g)|\overline{\mathcal{R}}|,\qquad g\in G,
\end{equation}
and $\cQ\colon L^2(X,{\rm d}x)\to L^2(\Xi,{\rm d}\xi)$ is a
partial isometry with
  \[
    \ker{\cQ}=\ker{\overline{\cR}}, \qquad \operatorname{Im}\cQ=\overline{\operatorname{Im}(\overline{\mathcal{R}})},
  \]
and  is semi-invariant with weight $\chi/|\chi|\equiv1$, i.e.\ \eqref{intertwiningU} is satisfied.
Since $\pi$ is irreducible, $\ker{\cQ}=\{0\}$ and it follows that   $\cQ$ is an isometry.

Define $
  W=\cQ|\overline{\cR}|\cQ^*$ with $\hat\pi$-invariant  domain
  \[
    \operatorname{dom}W=\{f\in L^2(\Xi,{\rm d}\xi) : \cQ^*f\in
    \operatorname{dom}\overline{\mathcal{R}}\} =\cQ (
      \operatorname{dom}\overline{\mathcal{R}}) \oplus
\overline{\operatorname{Im}(\overline{\mathcal{R}})}^\perp,
  \]
which is a densely defined positive operator in $L^2(\Xi,{\rm d}\xi)
$, semi-invariant with   weight~$\chi$. Indeed,
$\cQ (\operatorname{dom}\overline{\mathcal{R}})$  is dense in $\cQ(
  L^2(X,{\rm d}x))=
  \overline{\operatorname{Im}(\overline{\mathcal{R}})}$ since
$\overline{\cR}$ is densely defined by Lemma~\ref{lem:Rbar}. Observe that the $\hat\pi$-invariance of $\operatorname{dom}W$ follows from the $\pi$-invariance of $\operatorname{dom}\overline{\mathcal{R}}$. Further, by   \eqref{intertwiningU} and \eqref{intertwiningH}  and using that $\pi(g)$ is a unitary operator we readily derive
  \[
  \begin{split}
\hat\pi(g)W\hat\pi(g)^{-1}f&=  \hat\pi(g)\cQ|\overline{\cR}|\cQ^*\hat\pi(g)^{-1}f\\
&=\left(\hat\pi(g)\cQ\pi(g)^{-1}\right)\left(\pi(g)|\overline{\cR}|\pi(g)^{-1}\right)\left(\pi(g)\cQ^*\hat\pi(g)^{-1}\right)f\\
&=\cQ\left(\chi(g)|\overline{\cR}|\right)\cQ^*f\\
&=\chi(g)Wf,
  \end{split}
  \]
for every $f\in \operatorname{dom}W$. 

Since $\cQ^*\cQ=\operatorname{Id}$, then
  $ \overline{\mathcal{R}}= W\cQ$ and
 $\operatorname{Im}\overline{\mathcal{R}}\subset \operatorname{Im}W$. 
We   denote by  $\mathcal{I}$ the Moore-Penrose inverse of $W$
  \cite[Chapter 9, \S 3, Theorem~2]{2003generalizedinverses} with 
densely defined domain given by 
    \[ \operatorname{Im} W \oplus
      \operatorname{Im} W^\perp \supset
      \operatorname{Im} W\cQ = 
      \operatorname{Im}\overline{\mathcal{R}}.
    \]
Since $W$ is a positive operator in $L^2(\Xi,{\rm d}\xi) $, then
$\mathcal I$ is positive, too, and
  \begin{align*}
  & \mathcal{I} W f= f, \qquad f\in \operatorname{dom}W\cap
    \ker{W}^\perp,\\
    & W \mathcal{I}  f= f, \qquad f\in \operatorname{Im}W.
    \end{align*}
We claim that $\mathcal{I}$ is
  semi-invariant with weight $\chi^{-1}$ and
  \[ \mathcal{I}\overline{\mathcal{R}} f= \cQ f, \qquad f\in
    \operatorname{dom}\overline{\mathcal{R}}.\]
Indeed, if $f\in \operatorname{Im}W$, by definition $\mathcal{I}f=h$ with $h\in \operatorname{dom}W\cap
    \ker{W}^\perp$ and $Wh=f$. Thus, by the semi-invariance of $W$ we have that 
    \begin{align}\label{intertwiningI}
    \nonumber\hat\pi(g)\cI\hat\pi(g)^{-1}f&=\hat\pi(g)\cI\hat\pi(g)^{-1} Wh\\
    \nonumber&= \chi(g)^{-1} \hat\pi(g)\cI W\hat\pi(g)^{-1}h \\
     &= \chi(g)^{-1} \cI f,
    \end{align}
where we used  that $\hat\pi(g)^{-1}h\in \ker{W}^\perp$, which follows from the $\hat\pi$-invariance of $\ker{W}$. If $f\in \operatorname{Im}W^\perp$, by definition of $\cI$ the semi-invariance property \eqref{intertwiningI} is trivial.

Finally, since by \eqref{character} $\overline{\mathcal{R}}$ is an
injective operator, we have that  
$ \ker{W}=\ker\cQ^*$ and hence   $\ker
W^\perp=\overline{\rm{Im}\cQ}\supset\rm{Im}\cQ$, whence  $\cQ
f\in\operatorname{dom}W\cap 
    \ker{W}^\perp$ for any
    $f\in \operatorname{dom}\overline{\mathcal{R}}$. Therefore
    $\mathcal{I}\overline{\mathcal{R}} f = \cI W\cQ f=\cQ f$, as desired.
    
Assume now that $\hat{\pi}$ is irreducible. Since $\operatorname{Im}(\cQ)$ is
a $\hat{\pi}$-invariant closed subspace of $L^2(\Xi,{\rm d}\xi)$ by
\eqref{intertwiningU}, then 
$\cQ$ is surjective, so that $\cQ$ is unitary and $\pi$ and
$\hat{\pi}$ are equivalent by~\eqref{intertwiningU}. 
\end{proof}

It is worth observing that  the results of this section do not depend on the choice of the subspace $\cA$. Indeed, suppose to have  another $\pi$-invariant subspace
  \begin{equation}
    \label{eq:2bis}
 \cA' \subseteq \{ f\in L^2(X,{\rm d}x) :   f(\sigma(\xi)[\cdot])\in
   L^1(\hat{\xi}_0,m_0) \text{ a.e. }\xi\in\Xi,\;\mathcal R f\in L^2(\Xi, {\rm d}\xi)\}
\end{equation}
of $L^2(X,{\rm d}x)$, $\cA'\neq \{0\}$, such that $\mathcal{R}$ restricted
to $\mathcal A'$ is a closable operator. We denote by
$\overline{\mathcal{R}'}$ the closure of the restriction of $\cR$
  to $\cA'$. We prove that $\overline{\mathcal{R}'}$ actually
coincides with $\overline{\mathcal{R}}$. By Lemma~\ref{lem:Rbar}, the
closure $\overline{\mathcal{R}'}$ is a densely defined operator
satisfying 
\begin{equation}
\overline{\mathcal{R}'}\pi(g)f =\chi(g)^{-1}\hat{\pi}(g)
\overline{\mathcal{R}'} f, \qquad
f\in\operatorname{dom}(\overline{\mathcal{R}'}). 
\end{equation}
Thus,
by Theorem~\ref{duflomoore}, letting $\overline{\mathcal{R}'}=\cQ'
|\overline{\mathcal{R}'}|$ be the polar decomposition of
$\overline{\mathcal{R}'}$, we have that
$|\overline{\mathcal{R}'}|\colon\operatorname{dom}(\overline{\mathcal{R}'})\to
L^2(X,\D x)$ is a positive selfadjoint operator semi-invariant with
weight $|\chi|=\chi$ defined by \eqref{chi}. Therefore,
$|\overline{\mathcal{R}'}|$ and $|\overline{\cR}|$ are
proportional by Theorem~\ref{duflomoore}, item~(i). This implies that their
domains, $\operatorname{dom}(\overline{\mathcal{R}'})$ and
$\operatorname{dom}(\overline{\mathcal{R}})$,  coincide.
Hence, $\overline{\mathcal{R}'}=\overline{\mathcal{R}}$. 

\begin{exSIM}
Applying Lemma~\ref{lem:Rbar} to $\cR^\mathrm{pol}$, \gray{by
\eqref{polarintertwining}} its closure $\overline{\cR^{\rm 
    pol}}$ is a semi-invariant operator from  $\cA^\mathrm{pol}$ to
$L^2([0,\pi)\times\R)$ with weight $\chi(b,\phi,a)=a^{-1/2}$. 
By Theorem~\ref{ernestoduflomoore}  there exists  a positive
selfadjoint operator $\cI\colon\operatorname{dom}(\cI)\supseteq
 \operatorname{Im}(\overline{\cR^{\mathrm{pol}}}) \to L^2([0,\pi)\times\R)$, semi-invariant with
weight $\chi(g)^{-1}=a^{1/2}$, such that $\cI
\overline{\cR^\mathrm{pol}}$ extends to a unitary operator $\cQ\colon L^2(\R^2)\to L^2([0,\pi)\times\R)$ intertwining the
quasi-regular (irreducible) representations $\pi$ and $\hat{\pi}$. Hence 
\begin{alignat}{2}\label{factorizationpol}
  & \cI\mathcal{R}^{\text{pol}}f=\cQ f &&\qquad f\in\cA^{\rm pol}, \\
  & \cQ^*\cQ f= f &&\qquad  f \in L^2(\R^2),\nonumber \\
   & \cQ \cQ^* F= F &&\qquad  F \in L^2([0,\pi)\times\R),\nonumber \\
& \hat{\pi}(g)\,\cQ\,\pi(g)^{-1}=\cQ &&\qquad g\in SIM(2). \nonumber 
\end{alignat}

We can provide an explicit formula for $\cI$. Consider the subspace
\[
\cD=\{f\in L^2([0,\pi)\times\R):\int_{[0,\pi)\times\R}|\tau||(I\otimes\cF) f(\theta,\tau)|^2\ {\rm d}\theta{\rm d}\tau<+\infty\}
\]
and define the operator $\cJ\colon\cD\to L^2([0,\pi)\times\R)$ by
\[
(I\otimes\cF)\cJ f(\theta,\tau)=|\tau|^{\frac{1}{2}}(I\otimes\cF) f(\theta,\tau),
\]
a Fourier multiplier with respect to the last variable. A direct calculation shows that
$\cJ$ is a densely defined positive self-adjoint 
injective operator  and is semi-invariant with weight $\zeta(g)=\chi(g)^{-1}=a^{1/2}$. By Theorem~\ref{duflomoore},  item~(i), there exists $c>0$ such that $\cI= c
\cJ$ and we now show that $c=1$.
Consider a function $f\in\cA^{\rm pol}\setminus\{0\}$. Then, by Plancherel theorem
and the Fourier slice theorem \eqref{eq:FST} we have that 
\begin{align*}
\|f\|^2=\|\cI\mathcal{R}^{\text{pol}}f\|^2_{L^2([0,\pi)\times\R)}&=c^2\|(I\otimes\cF)\cJ\mathcal{R}^{\text{pol}}f\|^2_{L^2([0,\pi)\times\R)}\\
&=c^2\, \int_{[0,\pi)\times\R}|\tau||(I\otimes\cF)\mathcal{R}^{\text{pol}}f(\theta,\tau)|^2\D\theta\D\tau\\
&=c^2\, \int_{[0,\pi)\times\R}|\tau||\cF f(\tau w(\theta))|^2\D\theta\D\tau\\
&=c^2\|f\|^2.
\end{align*}
Thus, we obtain $c=1$.
\end{exSIM}

\section{Inversion of the Radon transform}\label{inversionsection}

 In this section, we make explicit use of the assumption that $\pi$
  is square-integrable to invert the Radon transform.  We recall that,
  under this assumption, there exists a self-adjoint operator
  \[C\colon\operatorname{dom}{C}\subseteq L^2(X,{\rm d}x) \to L^2(X,{\rm d}x), \]
semi-invariant with weight $\Delta^{\frac{1}{2}}$, where
  $\Delta$ is the modular function of $G$,  such that for all $\psi\in
\operatorname{dom}{C}$ with $\|C\psi\|=1$, the voice transform
$\mathcal V_\psi$ 
\[
(\mathcal V_\psi f)(g)= \scal{f}{\pi(g)\psi}, \qquad g\in G,
 \]
 is an isometry from $L^2(X,{\rm d}x)$ into $L^2(G)$
 and  we
have the weakly-convergent reproducing formula 
\begin{align}\label{recon}
f=\int_G   (\mathcal V_\psi f)(g)   \pi(g)\psi\ {\rm d}\mu(g),
\end{align}
where $\mu$ is the Haar measure 
(see, for example, \cite[Theorem~2.25]{fuhr05}). The vector $\psi$ is
called admissible vector.

As shown in the previous section, there exists a
positive self-adjoint operator $\cI$ semi-invariant with weight
$\chi^{-1}$ such that $\cI\mathcal{R}$ extends to  a unitary operator
$\mathcal{Q}$, which intertwines the quasi-regular representations $\pi$
and $\hat\pi$ of $G$ on $L^2(X,{\rm d}x)$ and $L^2(\Xi,{\rm d}\xi)$
respectively.

Since $\cQ$ is unitary and satisfies \eqref{intertwiningU}, the voice transform reads
 \begin{equation}\label{reconstruction}
\mathcal V_\psi f(g)= \langle f, \pi(g)\psi\rangle=\langle \mathcal{Q}f, \mathcal{Q}\pi(g)\psi\rangle
=\langle \mathcal{Q}f,\hat{\pi}(g)\mathcal{Q}\psi\rangle,
\qquad
g\in G,
\end{equation}
for all $f\in L^2(X,{\rm d}x)$.  Furthermore, the assumption that  $\pi$ is square-integrable ensures that any $f\in L^2(X,{\rm d}x)$ can be
reconstructed from its unitary Radon transform $\mathcal{Q}f$ by means
of the reconstruction formula \eqref{recon}, which becomes
\[
f=\int_G  \langle \mathcal{Q}f,\hat{\pi}(g)\mathcal{Q}\psi\rangle\,  \pi(g)\psi\ {\rm d}\mu(g).
\]  
Moreover, if  we can choose $\psi$ in such a way that
  $\mathcal{Q}\psi$ is in the domain of the operator $\cI$, by \eqref{reconstruction}, for
  all $f\in\operatorname{dom}\overline{\mathcal{R}}$, we have
\begin{align}\label{reconstructiontwo}
\nonumber\mathcal V_\psi f(g)
&=\langle \mathcal{Q}f,\hat{\pi}(g)\mathcal{Q}\psi\rangle\\
\nonumber &=\langle \cI\overline{\mathcal{R}}f,\hat{\pi}(g)\mathcal{Q}\psi\rangle\\
\nonumber &=\langle\overline{\mathcal{R}}f,\cI\hat{\pi}(g)\mathcal{Q}\psi\rangle\\
&=\chi(g)\langle\overline{\mathcal{R}}f,\hat{\pi}(g)\cI\mathcal{Q}\psi\rangle,
\end{align}
where we use that $\cI$ is a selfadjoint operator, semi-invariant with
weight $\chi^{-1}$.

By \eqref{reconstructiontwo} the voice transform  $\mathcal V_\psi f$ depends on $f$ only through its Radon transform
$\overline{\mathcal{R}}f$. Therefore, \eqref{reconstructiontwo} together with~\eqref{recon} 
allow to reconstruct an unknown
signal $f\in  \operatorname{dom}\overline{\mathcal{R}}   $ from its Radon transform. 
Explicitly, we have derived the following inversion formula for the Radon transform.
\begin{teo}\label{generalinversionformula}
 Let $\psi\in
L^2(X,{\rm d}x)$ be an admissible vector for $\pi$ such that
$\cQ\psi\in\operatorname{dom}\cI$, and set $\Psi= \cI\mathcal{Q}\psi$.
Then, for any $f\in \operatorname{dom}\overline{\mathcal{R}}$,
 \begin{equation}
    \label{inversionformula}
  f =  \int_G \chi(g)
  \langle \overline{\mathcal{R}} f,\hat{\pi}(g)\Psi\rangle \,     \pi(g)\psi\ {\rm d}\mu(g),
  \end{equation}
  where the integral is weakly convergent,and 
  \begin{equation}\label{reconstructionenergy}
  \|f\|^2=\int_G\chi(g)^2|\langle\overline{\mathcal{R}}f,\hat{\pi}(g)\Psi\rangle|^2{\rm d}\mu(g).
  \end{equation}
  If, in addition,  $\psi\in \operatorname{dom}\overline{\mathcal{R}}$, then 
$
\Psi=\cI^2\overline{\mathcal{R}}\psi. 
$
  \end{teo}
Note that the datum  $\overline{\mathcal{R}}f$ is
  analyzed by the family   $\{ \hat{\pi}(g)\Psi\}_{g\in G}$ and the
  signal $f$ is  reconstructed by a different family, namely $\{ \pi(g)\psi\}_{g\in G}$.

  \begin{exSIM}
It is known that $\pi$ is square-integrable and the
      corresponding voice transform gives rise to $2D$-directional wavelets \cite{anmu96}. 
An admissible vector is a function $\psi\in L^2(\R^2)$ satisfying the following admissibility condition \cite{anmu96}
\begin{equation}\label{admcond}
\int_{[0,2\pi)\times\R^{+}}|\cF{\psi}(A_aR_{\phi}^{{-1}}\xi)|^2{\rm d}\phi\frac{{\rm
      d}a}{a} =1,\qquad \text{for all $\xi\in\R^2/\{0\}$},
\end{equation}
which is equivalent to
\begin{equation}\label{admcondsim2}
\int_{\R^2}\frac{|\mathcal{F}\psi(\xi_1,\xi_2)|^2}{\xi_1^2+\xi_2^2}\D
  {\xi_1}\D
  {\xi_2}=1.
\end{equation}
Given $f\in\cA^{\rm pol}$, define $\cG(b,\phi,a)=a^{\frac{1}{2}}
    \langle\mathcal{R}{^{\rm pol}}f,\hat{\pi}(b,\phi,a)\Psi\rangle $, i.e.\ by \eqref{simrepresentationhat}
    \[
      \cG
      (b,\phi,a)=
     \int_{[0,\pi)\times\R }
      \mathcal{R}^{\text{pol}}f(\theta,t)\, \overline{\Psi\left(\theta-\phi\
        \text{mod}\ \pi,\frac{t-b\cdot \nt(\theta)}{a}\right)} {\rm
        d}\theta{\rm d}t.
    \]
    Then, taking into account that
    $\chi(b,\phi,a)=a^{-\frac{1}{2}}$,  \eqref{inversionformula} reads
    \begin{equation}\label{reconstructionsim}
  f (x) = \int_{\R^2\rtimes ([0,2\pi)\times\R^{+})}   \cG(b,\phi,a)
  \, \psi\big(R_{\phi}^{-1}\frac{x-b}{a}\big) \,
        \D b  \D \phi \frac{\D a}{a^5}.
\end{equation}
By \eqref{reconstructionenergy}, reconstruction formula \eqref{reconstructionsim} is equivalent to 
    \begin{equation}\label{polarizationsim}
\|f\|^2= \int_{\R^2\rtimes ([0,2\pi)\times\R^{+})} | \cG(b,\phi,a)|^2{\rm d}b{\rm d}\phi\frac{{\rm d}a}{a^5}.
  \end{equation} 
The idea to exploit the theory of the continuous wavelet transform to derive inversion formulae for the Radon transform is not new, we refer to \cite{holschneider91,berenstein-walnut-1994, madych-1999,walnut-1993,olson-destefano-1994}--to name a few. 

It is possible to obtain a version of  \eqref{polarizationsim} in which the scale parameter $a$ varies only in a compact set. Consider a smooth function $\Phi\in L^1(\R^2)\cap L^2(\R^2)$ such that 
\begin{equation}\label{Phi}
|\cF{\Phi}(\xi)|^2+\int_{[0,2\pi)\times(0,1)}|\cF{\psi}(A_aR_{\phi}^{-1}\xi)|^2{\rm d}\phi\frac{{\rm d}a}{a}=1.
\end{equation}
By Plancherel theorem, we have that 
\begin{align}\label{firstequality}
\nonumber\int_{\R^2}|\langle f,T_{b}\Phi\rangle|^2\ {\rm d}b&=\int_{\R^2}\left|\int_{\R^2}\cF f(\xi)\overline{\cF \Phi(\xi)}e^{2\pi ib\cdot\xi}\ {\rm d}\xi\right|^2\ {\rm d}b\\
\nonumber&=\int_{\R^2}|\cF^{-1}(\cF f\ \overline{\cF \Phi})(b)|^2\ {\rm d}b\\
&=\int_{\R^2}|\cF f(\xi)|^2|\cF \Phi(\xi)|^2\ {\rm d}\xi.
\end{align}
Using an analogous computation, by Plancherel theorem, equation \eqref{simrepresentationfrequency} and Fubini's theorem we have 
\begin{align}\label{secondequality}
\nonumber&\int_{\R^2\rtimes ([0,2\pi)\times(0,1))}| \cG(b,\phi,a)|^2{\rm d}b{\rm d}\phi\frac{{\rm d}a}{a^5}=\int_{\R^2\rtimes ([0,2\pi)\times(0,1))}|\langle f,\pi(b,\phi,a)\psi\rangle|^2{\rm d}b{\rm d}\phi\frac{{\rm d}a}{a^3}\\
\nonumber&=\int_{\R^2\rtimes ([0,2\pi)\times(0,1))}\left|\int_{\R^2} \cF f(\xi)\overline{\cF\psi(A_aR_{\phi}^{-1}\xi)}e^{2\pi ib\cdot\xi}\ {\rm d}\xi\right|^2{\rm d}b{\rm d}\phi\frac{{\rm d}a}{a}\\
\nonumber&=\int_{[0,2\pi)\times(0,1)}\left(\int_{\R^2}|\cF^{-1}(\cF f\overline{\cF\psi(A_aR_{\phi}^{-1}\cdot)})(b)|^2{\rm d}b\right){\rm d}\phi\frac{{\rm d}a}{a}\\
&=\int_{\R^2}|\cF f(\xi)|^2\left(\int_{[0,2\pi)\times(0,1)}|\cF\psi(A_aR_{\phi}^{-1}\xi)|^2{\rm d}\phi\frac{{\rm d}a}{a}\right){\rm d}\xi.
\end{align}
Thus, combining equations \eqref{Phi}, \eqref{firstequality} and \eqref{secondequality} we obtain the reconstruction formula 
\begin{equation}\label{reconstructionsim2}
\|f\|^2=\int_{\R^2}|\langle f,T_{b}\Phi\rangle|^2{\rm d}b+\int_{\R^2\rtimes ([0,2\pi)\times(0,1))}| \cG(b,\phi,a)|^2{\rm d}b{\rm d}\phi\frac{{\rm d}a}{a^5}.
\end{equation}

It is worth observing that there always exists a function $\Phi$
  satisfying~\eqref{Phi} provided that the admissible vector $\psi$
has fast Fourier decay. Indeed, if we require $\cF\psi$ to satisfy a
decay estimate of the form   
\[
|\cF{\psi}(\xi)|=O(|\xi|^{-L}),\qquad \text{for every $L>0$,}
\]
then, by \eqref{admcond} we have that 
\[
\begin{split}
z(\xi)&:=1 -
 \int_{[0,2\pi)\times(0,1)}|\cF{\psi}(A_aR_{\phi}^{-1}\xi)|^2{\rm d}\phi\frac{{\rm
  d}a}{a} \\
&=\int_{[0,2\pi)\times[1,+\infty)}|\cF{\psi}(A_aR_{\phi}^{-1}\xi)|^2{\rm d}\phi\frac{{\rm d}a}{a}\\
&\lesssim\int_{[0,2\pi)\times[1,+\infty)}a^{-2L}|\xi|^{-2L}\frac{{\rm d}a}{a}{\rm d}\phi\\
&\lesssim|\xi|^{-2L}.
\end{split}
\]
 Therefore, there exists a
smooth function $\Phi$ such that $\cF \Phi(\xi)=\sqrt{z(\xi)}$, so
that~\eqref{Phi} holds true. 

Finally, let us show that the first term in the right hand side of \eqref{reconstructionsim2} may be expressed in terms of $\cR^{\mathrm{pol}} f$ only, if we suppose that $\Phi\in\cA^{\rm pol}$.  We readily derive 
\begin{align}\label{coefficinetsPhi}
\nonumber\langle f,T_{b}\Phi\rangle= \langle f,
  \pi(b,0,1)\Phi\rangle&=\langle \cQ f, \cQ \pi(b,0,1)\Phi\rangle\\
\nonumber&=\langle \cQ  f,\hat{\pi}(b,0,1)\cQ \Phi\rangle\\
\nonumber&=\langle \cI\cR^{\rm pol} f,\hat{\pi}(b,0,1)\cI\cR^{\rm pol}\Phi\rangle\\
&=\langle \cR^{\rm pol} f,\hat{\pi}(b,0,1)\cI^2\cR^{\rm pol}\Phi\rangle,
\end{align}
where we observe that  $\cI\cR^{\rm pol}\Phi$ is always in the
domain of the operator $\cI$ since
\begin{align*}
\int_{[0,\pi)\times\R}|\tau||(I\otimes\cF)\cI\cR^{\text{pol}}\Phi(\theta,\tau)|^2{\rm d}\theta{\rm d}\tau&=\int_{[0,\pi)\times\R}|\tau|^2|(I\otimes\cF)\cR^{\text{pol}}\Phi(\theta,\tau)|^2{\rm d}\theta{\rm d}\tau\\
&=\int_{[0,\pi)\times\R}|\tau|^2|\cF \Phi(\tau w(\theta))|^2{\rm d}\theta{\rm d}\tau\\
&=\int_{\R^2}|\xi||\cF \Phi(\xi)|^2{\rm d}\xi<+\infty,
\end{align*} 
since by definition $\Phi$ is a smooth function. Therefore, reconstruction formula  \eqref{reconstructionsim2} reads
\[
\|f\|^2=\int_{\R^2}|\langle \cR^{\rm pol} f,\hat{\pi}(b,0,1)\cI^2\cR^{\rm pol}\Phi\rangle|^2{\rm d}b+\int_{\R^2\rtimes ([0,2\pi)\times(0,1))}| \cG(b,\phi,a)|^2{\rm d}b{\rm d}\phi\frac{{\rm d}a}{a^5},
\]
where all the coefficients depend on $f$ only through its polar Radon transform.

 It is worth observing that the domain of 
$\overline{\cR^{\text{pol}}}$ is related to the domain of $C$, which defines  the admissible vectors
of $\pi$.  By  Theorem \ref{duflomoore}, (ii), the
operator $|\overline{\cR^{\text{pol}}}|$ is a positive self-adjoint
operator semi-invariant with weight $\chi(b,\phi,a)=a^{-1/2}$, which is a power of the
modular function
$\Delta(b,\phi,a)=a^{-2}$,
i.e. $\chi(b,\phi,a)=\Delta(b,\phi,a)^{1/4}$. On the other hand,  $C$  is a positive self-adjoint
operator  semi-invariant with weight
$\Delta^{1/2}$
and  is such that $\psi\in L^2(\R^2)$ is an admissible vector of
the square-integrable representation $\pi$ if and only if
$\psi\in\operatorname{dom}C$ and $\|C\psi\|=1$. Therefore,
$|\overline{\cR^{\text{pol}}}|$ and $C$ are both positive
self-adjoint operators on $L^2(\R^2)$  semi-invariant
with a power of the modular function of $SIM(2)$ as weight. Finally,
consider the subspace  
\[
\cD_s=\{f\in L^2(\R^2): \int_{\R^2}|\xi|^{2s}|\cF f(\xi)|^2 \D\xi<+\infty\}
\]
of $L^2(\R^2)$. It is not difficult to verify that the Fourier multiplier $A_s\colon\cD_s\to L^2(\R^2)$ defined by
\begin{equation}\label{fouriermultipliers}
\cF A_sf(\xi)=|\xi|^{s}\cF f(\xi) 
\end{equation}
is a densely defined positive self-adjoint operator and 
is semi-invariant with weight
$\chi_s(b,\phi,a)=\Delta(b,\phi,a)^{-s/2}=a^s$. Thus, by
Theorem~\ref{duflomoore}, (i), the operators
$|\overline{\cR^{\text{pol}}}|$ and $C$ 
are given, up to a constant, by \eqref{fouriermultipliers} with
$s=-1/2$ and $s=-1$, respectively. The above argument explains
why the domain of $\cR^{\text{pol}}$ and the domain of $C$, and thus
the admissibility condition \eqref{admcond} of $\pi$, are strictly
related. A similar result can be proved for the examples
  illustrated in Section \ref{examples}. 
\end{exSIM}  

\section{Examples}\label{examples}

In this section, we illustrate two additional examples.

\subsection{The affine Radon transform and the shearlet transform}

\subsubsection{Groups and spaces}

The (parabolic) shearlet group $\mathbb{S}$ is the semidirect product
of $\R^2$ with the closed subgroup $K=\{N_sA_a\in {\rm GL}(2,\R):s\in\R,a\in\R^{\times}\}$ where 
\[
N_s=\left[\begin{matrix}1 & -s\\ 0 & 1\end{matrix}\right],\qquad A_a=a\left[\begin{matrix}1 & 0\\ 0 & |a|^{-1/2}\end{matrix}\right].
\]
We can identify the element $N_sA_a$ with the pair $(s,a)$ and we write $(b,s,a)$ for the elements in $\mathbb{S}$. With this identification the product law amounts to 
\[
(b,s,a)(b',s',a')=(b+N_sA_ab',s+|a|^{1/2}s',aa')
\]
and the inverse of $(b,s,a)$ is given by 
\[
(b,s,a)^{-1}=(-A_a^{-1}N_s^{-1}b,-|a|^{-1/2}s,a^{-1}).
\]
A left Haar measure of $\mathbb{S}$ is 
\[
{\rm d}\mu(b,s,a)=|a|^{-3}{\rm d}b{\rm d}s{\rm d}a, 
\]
with ${\rm d}b$, ${\rm d}s$ and ${\rm d}a$ the Lebesgue measures on $\R^2$, $\R$ and $\R^{\times}$, respectively. The shearlet group acts transitively on $\Xi=\R\times\R$ by the action 
\[
(b,s,a)^{-1}.(v,t)=\left(|a|^{-1/2}(v-s),\frac{t-n(v)\cdot b}{a}\right)
\]
where $n(v)={^t(1,v)}$. The isotropy at $\xi_0=(0,0)$ is 
\[
H=\{((0, b_2),0,a): b_2\in\R,a\in\R^{\times}\},
\]
so that $\Xi=\mathbb{S}/H$. It is immediate to verify that the Lebesgue measure $\D\xi=\D v\D t$ is a relatively invariant measure on $\Xi$ with positive character $\beta(b,s,a)=|a|^{3/2}$. Now, we consider the sections $s\colon\R^2\to \mathbb{S}$ and $\sigma\colon\R\times\R\to\mathbb{S}$ defined by
\[
s(x)=(x,0,1),\qquad \sigma(v,t)=((t,0),v,1).
\]
Thus, we have that 
\begin{align*}
\hat{\xi}_0&=H[x_0]=\{(0, b_2): b_2\in\R\}\simeq\R,\\
 \check{x}_0&=K.\xi_0=\{(s,0):s\in\R\}\simeq\R.
\end{align*}
It is easy to check that the Lebesgue measure $\D b_2$ on  $\hat{\xi}_0$ is a relatively $H$-invariant measure with $\gamma((0, b_2),0,a)=|a|^{1/2}$. Further, we can compute 
\[
\widehat{(v,t)}=\sigma(v,t)[\hat{\xi}_0]=\{x\in\R^2:x\cdot n(v)=t\},
\]
which is the set of all points laying on the line of equation $x\cdot n(v)=t$ and 
\[
\check{x}=s(x).\check{x}_0=\{(v,t)\in \R\times\R:t-n(v)\cdot x=0\},
\]
which parametrizes the set of all lines passing through the point $x$ except the horizontal one. Thus, the maps $x\mapsto\check{x}$ and $(v,t)\mapsto\widehat{(v,t)}$ are both injective. Therefore, $X=\R^2$ and $\Xi=\R\times\R$ are homogeneous spaces in duality.

\subsubsection{The representations}

The (parabolic) shearlet group $\mathbb{S}$ acts on $L^2(\R^2)$ via the shearlet representation, namely
\begin{equation}\label{shearletrepresentation}
\pi(b,s,a)f(x)=|a|^{-3/4}f(A_a^{-1}N_s^{-1}(x-b)).
\end{equation}
It is well known that the shearlet representation is
irreducible \cite{dahlke2008}.

Furthermore, since $\beta(b,s,a)=|a|^{3/2}$, the group $\mathbb{S}$ acts on $L^2(\R\times\R,\D v\D t)$ by means of the quasi-regular representation $\hat{\pi}$ defined by
\begin{equation}\label{shearletrepresentationhat}
\hat{\pi}(b,s,a)F(v,t)=|a|^{-\frac{3}{4}}F\left(|a|^{-1/2}(v-s),\frac{t-n(v)\cdot b}{a}\right).
\end{equation}
By Mackey imprimitivity theorem \cite{folland16}, one can show
that also $\hat{\pi}$ is irreducible.

\subsubsection{The Radon transform}

 By \eqref{generalradon}, the Radon transform between the homogeneous spaces in duality $\R^2$ and $\R\times\R$ is defined as 
\begin{equation}\label{radonaff}
\mathcal{R}^{\text{aff}}f(v,t)=\int_{\R}f(t-vy,y){\rm d}y,
\end{equation}
which is the so-called affine Radon transform
\cite{coeagula10,gr11}. 

Following the same arguments as in Example \ref{ex:SIM}, 
we define
\[
\cA^{\rm aff}=\{f\in L^1(\R^2)\cap L^2(\R^2):\int_{\R^2}\frac{|\cF f(\xi)|^2}{|\xi_1|}\D\xi<+\infty\},
\]
where $\xi=(\xi_1,\xi_2)\in\R^2$, 
which is $\pi$-invariant and is such that $\cR^{\text{aff}}f\in
L^2(\R\times\R)$ for all $f\in\cA^{\rm aff}$  (we refer to
\cite{bardemadeviodo} for more details). Furthermore, as in Example \ref{ex:SIM},
  it is easy to show that $\cR^{\rm aff}$, regarded as operator from
  $\cA^{\rm aff}$ to $L^2(\R\times\R)$, is closable.
  
\subsubsection{The Unitarization theorem}
Since $\beta(b,s,a)={|a|^{3/2}}$, $\gamma(b,s,a)=|a|^{1/2}$ and $|\det(N_sA_a)|=|a|^{3/2}$ the affine Radon transform satisfies the intertwining property
\[
\mathcal{R}^{\text{aff}}\pi(b,s,a)=\chi(b,s,a)^{-1}\hat{\pi}(b,s,a)\mathcal{R}^{\text{aff}},
\]
where $\chi(b,\phi,a)^{-1}=|a|^{1/2}$.

By Lemma~\ref{lem:Rbar}, the closure $\overline{\cR^{\rm aff}}$
of  the affine Radon transform is a semi-invariant operator
 with weight
$\chi(b,s,a)=|a|^{-1/2}$. Therefore, by
Theorem~\ref{ernestoduflomoore},
there exists a positive selfadjoint operator $\cI\colon
  \operatorname{dom}(\cI)\subseteq L^2(\R\times\R)\to
  L^2(\R\times\R)$ semi-invariant with weight $\zeta(g)=\chi(g)^{-1}=|a|^{1/2}$
  such that $\cI\mathcal{R}^{\text{aff}}$ extends to a unitary operator
  $\cQ$ from $L^2(\R^2)$ onto $L^2(\R\times\R)$,  which intertwines
  the quasi-regular (irreducible) representations $\pi$ and
  $\hat{\pi}$. Reasoning as in Example~\ref{ex:SIM}, it is possible to
  show that
   the operator $\cI$ is defined by
\[
(I\otimes\cF)\cI f(\omega,\tau)=|\tau|^{\frac{1}{2}}(I\otimes\cF) f(\omega,\tau),\qquad f\in\cD,
\]
where
\[
\cD=\{f\in L^2(\R\times\R):\int_{\R\times\R}|\tau||(I\otimes\cF) f(\omega,\tau)|^2\ {\rm d}\omega{\rm d}\tau<+\infty\}.
\]

\subsubsection{The inversion formula}

It is known that the shearlet representation $\pi$ is
square-integrable and its admissible vectors are the functions $\psi$ in
$L^2(\R^2)$ satisfying  
\begin{equation}\label{admissibleconditionshearlet}
\int_{\R^2}\frac{|\cF\psi(\xi)|^2}{|\xi_1|^2}\D\xi=1,
\end{equation}
where $\xi=(\xi_1,\xi_2)\in\R^2$ \cite{dahlke2008}.
The shearlet transform is then $\cS_{\psi}f(b,s,a)=\langle f,\pi(b,s,a)\psi\rangle$, and is a multiple of an isometry from $L^2(\R^2)$ into $L^2(\mathbb{S},{\rm d}\mu)$ provided that $\psi$ satisfies the admissible condition \eqref{admissibleconditionshearlet}. 
By Theorem~\ref{generalinversionformula}, for any $f\in\cA^{\rm aff}$
we have the reconstruction formula 
\begin{equation}\label{reconstructionformulashearlet}
 f =  \int_{\R^2\times\R\times\R^{\times}}
 \cS_{\psi}f(b,s,a) \,     \pi(b,s,a)\psi\ \frac{\D b\D s\D a}{|a|^3},
\end{equation}
where the coefficients $ \cS_{\psi}f(b,s,a)$ are given by
\[
\cS_{\psi}f(b_1,b_2,s,a)=|a|^{-5/4}\int_{\R\times\R}\cR^{\text{aff}}f(v,t) \overline{\Psi\left(\frac{v-s}{|a|^{1/2}},\frac{t-n(v)\cdot b}{a}\right)}\D v\D t.
\]
If we choose $\Psi$ such that $\Psi(v,t)=\Psi_2(v)\Psi_1(t)$, then 
\begin{equation}\label{shearletcoefficients}
\cS_{\psi}f(b_1,b_2,s,a)=|a|^{-3/4}\int_{\R}\cW_{\Psi_1}(\cR^{\text{aff}}f(v,\cdot))(n(v)\cdot b,a) \overline{\Psi_2\left(\frac{v-s}{|a|^{1/2}}\right)}\D v,
\end{equation}
provided that $\Psi_1$ is a 1D-wavelet.

This argument gives an alternative proof of Theorems 8 and
10 in \cite{bardemadeviodo}, where it is also proved that formulas \eqref{shearletcoefficients} can actually be extended to the whole $L^1(\R^2)\cap L^2(\R^2)$ and it is not difficult to verify that this extension works in the same way for  Example \ref{ex:SIM}. Formula \eqref{reconstructionformulashearlet} is a continuous version of the reconstruction formula presented  in \cite[Theorem~3.3]{coeagula10}. 

\subsection{The spherical means Radon transform}

\subsubsection{Groups and spaces}

Take the same group $G=SIM(2)$ as in Example~\ref{ex:SIM},
namely $G=\R^2\rtimes K$, with $K=\{R_\phi A_a\in \text{GL}(2,\R):\phi\in
[0,2\pi),\;a\in\R^+\}$.  We consider the space
$\Xi=\R^2\times \R^+$, which we think of as parametrizing centers and radii of circles in
$\R^2$, with the action 
\begin{equation}\label{eq:actioncircle}
(b,\phi,a).(\ct,r)=(b+aR_\phi \ct,ar).
\end{equation}
An immediate calculation shows that the isotropy at $\xi_0=((1,0),1)$ is
\[
H=\{\bigl((1-\cos\phi,-\sin\phi),\phi,1\bigr):\phi\in [0,2\pi)\}.
\]
By direct computation, recalling that $x_0=0$,
\[
\hat\xi_0=H[x_0]=\{(1-\cos\phi,-\sin\phi):\phi\in [0,2\pi)\}
\]
is the circle with center $(1,0)$ and radius $1$ and
\[
\check{x}_0=K.\xi_0=\{((a\cos{\phi},a\sin{\phi}),a):\phi\in[0,2\pi),\, a\in\R^+\}
\]
is the family of circles passing through the origin.  The measure $\D m_0=\D\phi$ is $H$-invariant on $\hat\xi_0$, since the action of $H$ on $\hat\xi_0$ is given by a simple rotation of a fixed angle. This gives $\gamma(h)\equiv 1$.

We define the section $\sigma\colon\Xi\to SIM(2)$ by $\sigma(\ct,r)=(\ct-(r,0),0,r)$. Thus, for $\xi=(\ct,r)\in\Xi$ we have
\begin{equation}\label{eq:circles}
\hat\xi=\sigma(\ct,r)[\hat\xi_0]=\{\ct-r\nt(\phi):\phi\in [0,2\pi)\},
\end{equation}
namely, the circle with center $\ct$ and radius $r$ and, for $x\in\R^2$ we have
\[
\check{x}=s(x).\check{x}_0=\{(x+(a\cos{\phi},a\sin{\phi}),a):\phi\in[0,2\pi),\, a\in\R^+\},
\]
that is, the family of circles passing through the point $x$. It is
easy to see that the maps $x\mapsto\check{x}$ and $\xi\mapsto\hat\xi$
are both injective. Thus, $X=\R^2$ and $\Xi=\R^2\times \R^+$ are
homogeneous spaces in duality. 

We now fix a relatively invariant
measure on $\Xi$: as we will show, this requires some care. Given $\alpha\in\R$,
 we have
\[
\int_{\R^2\times \R^+} f\left((b,\phi,a)^{-1}.(\ct,r)\right)\D \ct\frac{\D r}{r^\alpha}=a^{3-\alpha}\int_{\R^2\times \R^+} f(\ct,r)\D \ct\,\frac{\D r}{r^\alpha},
\]
so that the measure $\D\xi=\D \ct\frac{\D r}{r^\alpha}$ is a
relatively invariant measure on $\Xi$ with character
$\beta(b,\phi,a)=a^{3-\alpha}$.

\subsubsection{The representations}
Since the group $G$ is the same as in Example~\ref{ex:SIM}, the representation $\pi$ is given by
\eqref{simrepresentation},  whereas we have to compute the
quasi-regular representation $\hat\pi$ acting on $L^2(\Xi,\D\xi)$. Since
$\beta(b,\phi,a)=a^{3-\alpha}$, by \eqref{eq:inverseSIM} and
\eqref{eq:actioncircle} we have 
\begin{equation}\label{eq:hatcircle}
\begin{split}
\hat{\pi}(b,\phi,a)F(\ct,r)&=a^{\frac{\alpha-3}{2}}F((-A_a^{-1}R_{\phi}^{-1}b,-\phi\ \text{mod}\ 2\pi, a^{-1}).(\ct,r))\\
&=a^{\frac{\alpha-3}{2}}F( a^{-1}R_{-\phi}(\ct-b),a^{-1}r),
\end{split}
\end{equation}
which is irreducible by Mackey imprimitivity theorem \cite{folland16}. 

\subsubsection{The Radon transform}

By \eqref{eq:circles} and \eqref{generalradon}, the Radon trasform in this case is given  by
\[
\mathcal{R}^{\rm cir}f(\ct,r)=\int_0^{2\pi}f(\ct-r\nt(\phi))\D\phi,
\]
namely, the integral of $f$ over the circle of center $\ct$ and radius
$r$.  This is the so-called spherical means Radon transform
\cite{2008-kuchment}. It is worth observing that more interesting
problems arise when the available  centers and radii are restricted to
some hypersurface: this does not easily fit into our assumptions, and
it is left for future investigation. 

Let us now determine a suitable $\pi$-invariant subspace $\cA^{\rm
  cir}$ of $L^2(\R^2)$ as in \ref{ass:A}. In order to do that, it is
useful to derive a Fourier slice theorem for $\mathcal{R}^{\rm
  cir}$. For $f\in L^1(\R^2)\cap L^2(\R^2)$, by Fubini's theorem and
using \cite[Eq.~10.9.1]{NIST:DLMF}, we have 
\[
\begin{split}
(\cF\otimes I) \cR^{\rm cir} f(\tau,r)&=\int_0^{2\pi}\int_{\R^2} f(\ct-r\nt(\phi))e^{-2\pi i \ct\cdot\tau}\D \ct\D\phi\\
&=\int_0^{2\pi} e^{-2\pi i r \nt(\phi)\cdot\tau}\D\phi \int_{\R^2} f(\ct)e^{-2\pi i \ct\cdot\tau}\D \ct \\
&=2\pi J_0(2\pi|\tau|r)\cF f(\tau),
\end{split}
\]
where $J_0$ is the Bessel function of the first kind. As a consequence, by Plancherel theorem, recalling that $\D\xi=\D \ct\frac{\D r}{r^{\alpha}}$ we obtain
\[
\norm{\cR^{\rm cir} f}_{L^2(\Xi,\D\xi)}^2=
\norm{(\cF\otimes I) \cR^{\rm cir} f}_{L^2(\Xi,\D \tau\frac{\D r}{r^{\alpha}})}^2=c_\alpha \int_{\R^2}|\cF f(\tau)|^2|\tau|^{\alpha-1}\,\D\tau,
\]
where 
\begin{equation}\label{calpha}
c_\alpha=(2\pi)^{\alpha+1}\int_{\R^+}\frac{|J_0(r)|^2}{r^\alpha}\D r.
\end{equation}
Observe that $c_\alpha$ is finite if and only if  $\alpha\in
  (0,1)$, so that from now on we assume that $\alpha\in
  (0,1)$ and we set
\[
\cA^{\rm cir}_{\alpha}=\{f\in L^1(\R^2)\cap L^2(\R^2): \int_{\R^2}|\cF f(\tau)|^2|\tau|^{\alpha-1}\,\D\tau<+\infty\},
\]
which is $\pi$-invariant and is such that $\cR^{\text{cir}}f\in
L^2(\Xi,\D\xi)$ for all $f\in\cA^{\rm cir}_{\alpha}$. Furthermore, as in Example~\ref{ex:SIM},
  it is easy to show that $\cR^{\rm cir}$, regarded as operator from
  $\cA^{\rm cir}_{\alpha}$ to $L^2(\Xi,\D\xi)$, is
  closable. We stress that, if $\alpha\notin (0,1)$,  the set
    \[ \{ f\in L^2(X,\D x): \cR^{\rm cir}f \in L^2(\Xi,\D\xi)\} =\{0\},\]
i.e.\ it is trivial. This motivates  the role of~Assumption~\ref{ass:A} in
our construction.

\subsubsection{The Unitarization theorem}
  
By $\eqref{intertwining}$ and $\eqref{chi}$ we have that 
\[
\mathcal{R}^{\text{cir}}\pi(b,\phi,a)=a^\frac{1-\alpha}{2}\hat{\pi}(b,\phi,a)\mathcal{R}^{\text{cir}},
\]
since $\beta(b,\phi,a)=a^{3-\alpha}$, $\gamma(b,\phi,a)=1$ and
$\det(R_\phi A_a)=a^2$, and so $\chi(b,\phi,a)=a^\frac{\alpha-1}{2}$.  

By Theorem~\ref{ernestoduflomoore}, there exists a positive self-adjoint operator $\cI$, semi-invariant with weight $a^\frac{1-\alpha}{2}$, such that $\cI\cR^{\rm cir}$ extends to a unitary operator $\cQ\colon L^2(\R^2)\to L^2(\Xi,\D\xi)$. Moreover, $\cQ$ intertwines $\pi$ and $\hat\pi$, namely
\begin{equation*}
\hat{\pi}(b,\phi,a)\,\cQ\,\pi(b,\phi,a)^{-1}=\cQ,
\qquad (b,\phi,a)\in SIM(2).
\end{equation*}
As in the other examples, by using Theorem~\ref{duflomoore}, part (i), it is possible to show that there exists a constant $k_\alpha\in\R^+$ such that $\cI=k_\alpha\cJ$ with
\[
(\cF\otimes I)\cJ f(\tau,r)=|\tau|^{\frac{1-\alpha}{2}}(\cF\otimes I) f(\tau,r),\qquad f\in\cD,
\]
where
\[
\cD=\{f\in L^2(\Xi,\D\xi):\int_{\R^2\times\R^+}|\tau|^{1-\alpha}|(\cF\otimes I) f(\tau,r)|^2\ {\rm d}\tau\frac{\D r}{r^\alpha}<+\infty\}.
\]
Using the same argument as in Example \ref{ex:SIM}, it is possible to determine the constant $k_\alpha$. Take a function $f\in \cA^{\rm cir}_\alpha\setminus\{0\}$. By Plancherel theorem and the Fourier slice theorem obtained for $\cR^{\rm cir}$ we have that
\begin{align*}
\|f\|^2=\|\cI\mathcal{R}^{\text{cir}}f\|^2_{L^2(\R^2\times\R^+)}&=k_\alpha^2\|(\cF\otimes I)\cJ\mathcal{R}^{\text{cir}}f\|^2_{L^2(\R^2\times\R^+)}\\
&=k_{\alpha}^2\, \int_{\R^2\times\R^+}|\tau|^{1-\alpha}|(\cF\otimes I)\mathcal{R}^{\text{cir}}f(\tau,r)|^2\D\tau\frac{\D r}{r^\alpha}\\
&=k_\alpha^2 c_\alpha\|f\|^2,
\end{align*}
where $c_\alpha$ is given by \eqref{calpha}.
Thus, we obtain that $k_\alpha=c_\alpha^{-1/2}$.

\subsubsection{The inversion formula}

Applying Theorem~\ref{generalinversionformula} to $\cR^{\rm cir}$ we obtain the inversion formula for $f\in\cA^{\rm cir}_{\alpha}$
 \begin{equation*}
  f =  \int_{SIM(2)}  a^{\frac{\alpha-9}{2}}
  \langle \mathcal{R}^{\rm cir} f,\hat\pi(b,\phi,a)\Psi\rangle_{L^2(\Xi,\D\xi)} \,     \psi(R_{-\phi}\frac{x-b}{a})\ \D b\D \phi\D a,
  \end{equation*}
  where we used that $\chi(b,\phi,a)=a^\frac{\alpha-1}{2}$, the expression for the Haar measure of $SIM(2)$ given in \eqref{eq:haarSIM} and the expression for the representation $\pi$ given in \eqref{simrepresentation}. 

\bibliographystyle{abbrv}
\bibliography{biblio2}
\end{document}